\documentclass{amsart}
\usepackage{amsfonts}
\usepackage{amsmath,amssymb}
\usepackage{amsthm}
\usepackage{amscd}
\usepackage{graphics}
\usepackage{graphicx}
\theoremstyle{definition}{
\newtheorem{Def}{{\rm Definition}}
\newtheorem{Ex}{{\rm Example}}
\newtheorem{Rem}{{\rm Remark}}

}
\theoremstyle{plain}
{
\newtheorem{Cor}{Corollary}
\newtheorem{Prop}{Proposition}
\newtheorem{Thm}{Theorem}
\newtheorem{MainThm}{Main Theorem}

}

\begin{document}
\title[$3$-dimensional complex projective spaces have no special generic maps]{Proofs of the non-existence of special generic maps on the $3$-dimensional complex projective space}
\author{Naoki Kitazawa}
\address{Institute of Mathematics for Industry, Kyushu University, 744 Motooka, Nishi-ku Fukuoka 819-0395, Japan\\
 TEL (Office): +81-92-802-4402 \\
 FAX (Office): +81-92-802-4405 \\
}
\email{n-kitazawa@imi.kyushu-u.ac.jp, naokikitazawa.formath@gmail.com}
\urladdr{https://naokikitazawa.github.io/NaokiKitazawa.html}

\begin{abstract}
We prove the non-existence of {\it special generic} maps on $3$-dimensional complex projective space as our new result and a corollary by several methods.
{\it Special generic} maps are generalizations of Morse functions with exactly two singular points on spheres and  
canonical projections of unit spheres are special generic. 
Our paper focuses on such maps on closed and simply-connected manifolds of classes containing the $3$-dimensional complex projective space.

The differentiable structures of spheres admitting special generic maps are known to be restricted strongly. Special generic maps on closed and simply-connected manifolds and projective spaces have been studied by various people including the author. 
The (non-)existence and construction are main problems. Studies on such maps on closed and simply-connected manifolds whose dimensions are greater than $5$ have been difficult.


\end{abstract}


\maketitle
\section{Introduction.}
\label{sec:1}
  
What are {\it special generic} maps? In short, Morse functions on spheres with exactly two singular points are regarded as simplest special generic maps. 
They play important roles in so-called Reeb's theorem, characterizing spheres topologically via Morse functions (Theorem \ref{thm:1} (\ref{thm:1.0})).
Canonical projections of unit spheres are also special generic.

We define a {\it special generic} map. 

First we introduce terminologies and notation on (smooth) manifolds and maps,
For a positive integer $k$, ${\mathbb{R}}^k$ denotes the $k$-dimensional Euclidean space where $\mathbb{R}:={\mathbb{R}}^1$. We regard this as a natural smooth manifold with a standard Euclidean metric as a Riemannian metric. For $x \in {\mathbb{R}}^k$, $||x||$ denotes the distance between $x$ and the origin $0 \in {\mathbb{R}}^k$. For a positive integer $k$, $D^k:=\{x \mid ||x|| \leq 1\} \subset {\mathbb{R}}^k$ denotes the $k$-dimensional unit disk, which is a $k$-dimensional compact, connected and smooth closed submanifold. $S^{k-1}:=\{x \mid ||x||=1\} \subset {\mathbb{R}}^{k}$ denotes the ($k-1$)-dimensional unit sphere, which is a ($k-1$)-dimensional closed smooth submanifold with no boundary. 
The $0$-dimensional unit sphere is a two-point set endowed with the discrete topology where $k=1$ and $k-1=0$. The ($k-1$)-dimensional unit sphere is connected for $k-1 \geq 1$.

It is well-known that every smooth manifold has the structure of a canonical PL manifold. We regard smooth manifolds as the canonical PL manifolds. It is also known that topological manifold is regarded as a CW complex.

For a manifold, a polyhedron $X$, or a CW complex, for example, let $\dim X$ denote the dimension of $X$, which is well-defined.

Let $c:X \rightarrow Y$ be a smooth map from a smooth manifold $X$ into another smooth manifold $Y$. A {\it singular} point $p \in X$ of a smooth map $c:X \rightarrow Y$ is a point where the rank of the differential ${dc}_p$ is smaller than $\min\{\dim X,\dim Y\}$. The {\it singular set} $S(c)$ of $c$ is defined as the set of all singular points of $c$.

A {\it diffeomorphism} is defined as a smooth map between two manifolds which is a homeomorphism with no singular points. A diffeomorphism from a manifold $X$ onto the same manifold is said to be a {\it diffeomorphism} on $X$. The diffeomorphism group of a smooth manifold $X$ is the group of all diffeomorphisms on $X$, topologized with the so-called {\it Whitney $C^{\infty}$ topology}. Whitney $C^{\infty}$ topologies are well-known to be natural topologies on sets of smooth maps between two manifolds. Consult \cite{golubitskyguillemin} for example. 

Two smooth manifolds are {\it diffeomorphic} if there exists a diffeomorphism from a manifold to the other manifold.
A {\it homotopy sphere} means a smooth manifold which is homeomorphic to a sphere. A {\it standard sphere} means a homotopy sphere which is diffeomorphic to a unit sphere. An {\it exotic sphere} means a homotopy sphere which is not a standard sphere. A {\it standard disk} is a smooth manifold diffeomorphic to a unit disk. 

\cite{milnor} is a pioneering  paper on an exotic sphere. This shows $7$-dimensional exotic spheres. \cite{kervairemilnor} is on homotopy spheres whose dimensions are greater than $4$. In short, exotic spheres are classified via abstract algebraic topological theory in dimensions greater than $4$. It is well-known that exotic spheres do not exist in dimensions $1$, $2$, $3$, $5$ and $6$. The existence or the non-existence of $4$-dimensional exotic spheres is an open problem. A smooth manifold which is homeomorphic to a unit disk and which is not a standard disk is still undiscovered. If it exists, then it must be $4$-dimensional.    

\begin{Def}
\label{def:1}
A smooth map $c:X \rightarrow Y$ from a smooth manifold $X$ with no boundary into another smooth manifold $Y$ with no boundary satisfying $\dim X \geq \dim Y$ is a {\it special generic} map if at each singular point $p$, it is locally represented by the form
$$(x_1, \cdots, x_{\dim X}) \mapsto (x_1,\cdots,x_{\dim Y-1},\sum_{j=1}^{\dim X-\dim Y+1} {x_{j+\dim Y-1}}^2)$$
for suitable coordinates. 
\end{Def}

Interestingly, the differentiable structures of spheres admitting special generic maps are known to be restricted strongly. 
As Example \ref{ex:1} shows later, manifolds seen as elementary such as ones represented as connected sums of finitely many manifolds represented as the products of two spheres admit such maps in considerable cases.
The existence or the non-existence and construction of special generic maps on closed and simply-connected manifolds and projective spaces have been studied by various people including the author. Our main theorem is as follows and our main result is proofs of this new fact. 

\begin{MainThm}
	\label{mthm:1}
The $3$-dimensional complex projective space does not admit {\it special generic maps} into Euclidean spaces.
\end{MainThm}

This manifold is also a $6$-dimensional smooth closed and simply-connected manifold. Note that classifications of closed and simply-connected manifolds whose dimensions are greater than $4$ have a long history. \cite{jupp, wall, zhubr, zhubr2} are on classifications of $6$-dimensional closed and simply-connected manifolds. 

The organization of our paper is as follows. In the second section, we review existing studies on special generic maps such as fundamental properties and the existence or the non-existence of special generic maps on projective spaces and closed and simply-connected manifolds, which are presented before shortly. It may be important to review special generic maps on homotopy spheres. However we omit this considering the main content of the present paper. We only introduce \cite{calabi, saeki, saeki2, wrazidlo, wrazidlo2} and preprints \cite{kitazawa1, kitazawa2, kitazawa3, kitazawa4} by the author, which review related theory. In the third section,
we review topological properties of complex projective spaces.
In the fourth section, we show Main Theorem \ref{mthm:1} as corollaries by three methods. As another result related to the third method, we have another main theorem.

\begin{MainThm}
\label{mthm:2}
There exists a family $\{M_{\lambda}\}_{\lambda \in \Lambda}$ of countably many $6$-dimensional {\rm Bott} manifolds enjoying the following two properties.
\begin{enumerate}
	\item $M_{{\lambda}_1}$ and $M_{{\lambda}_2}$ are not homeomorphic for any distinct two elements ${\lambda}_1, {\lambda}_2 \in \Lambda$.
	\item $M_{\lambda}$ does not admit special generic maps into ${\mathbb{R}}^n$ for any $\lambda \in \Lambda$ and $n=1,2,3,4,5,6$.
\end{enumerate}
\end{MainThm}

Shortly, the class of ({\it generalized}){\it Bott manifolds} is an important class of closed and simply-connected smooth manifolds which are so-called {\it toric} manifolds. 
The products of finitely many $2$-dimensional (standard) spheres are Bott manifolds. Complex projective spaces are generalized Bott manifolds and they are not Bott-manifolds if they are not diffeomorphic to a $2$-dimensional sphere or the $1$-dimensional complex projective space. 
Related exposition and our proof are presented in Example \ref{ex:2} (\ref{ex:2.2}).

The fifth section is devoted to related remarks. We also present another main theorem as Main Theorems \ref{mthm:3} and \ref{mthm:4}.

The present paper includes similar descriptions as some of descriptions in \cite{kitazawa3} for example. \cite{kitazawa3} also studies special generic maps on closed and simply-connected manifolds whose dimensions are greater than $5$ or $6$. Manifolds whose dimensions are at least $5$ are said to be higher dimensional in general. 

\section{Some existing studies on special generic maps.}
\begin{Prop}
\label{prop:1}
For a special generic map in Definition \ref{def:1}, we have the following properties.
\begin{enumerate}
\item The singular set is an {\rm (}$n-1${\rm )}-dimensional smooth closed submanifold with no boundary.
\item The restriction of the map to the singular set is a smooth immersion.
\item For suitable coordinates, around each singular point, the map is represented as the product map of a Morse function and the identity map on a small open neighborhood of the singular point where the neighborhood is considered in the singular set.
\end{enumerate}
\end{Prop}

Hereafter, we use terminologies and notions on bundles such as {\it fibers}, {\it structure groups}, {\it projections}, {\it sections} and {\it trivial bundles} without rigorous expositions for example.
A bundle is said to be a {\it smooth} bundle if its fiber is a smooth manifold and its structure group is a subgroup of the diffeomorphism group. A smooth bundle is a {\it linear} bundle whose fiber is an Euclidean space, a unit sphere, or a unit disk and whose structure group consists of linear transformations.
We also explicitly or implicitly apply arguments on bundles and characteristic classes of linear bundles. For linear bundles and more general bundles, see \cite{milnorstasheff,steenrod} for example. 

\begin{Prop}[E. g. \cite{saeki}]
\label{prop:2}
Let $m \geq n \geq 1$ be integers.
For a special generic map $f:M \rightarrow N$ on an $m$-dimensional closed manifold $M$ into an $n$-dimensional manifold $N$ with no boundary, the following properties are enjoyed.
\begin{enumerate}
\item \label{prop:2.1}
There exists an $n$-dimensional compact manifold $W_f$ and a smooth immersion $\bar{f}:W_f \rightarrow N$.
\item \label{prop:2.2}
There exists a smooth surjection $q_f:M \rightarrow W_f$ and $f=\bar{f} \circ q_f$.
\item \label{prop:2.3}
$q_f$ maps the singular set $S(f)$ of $f$ onto the boundary $\partial W_f \subset W_f$ as a diffeomorphism.
\item \label{prop:2.4}
 We have the following two bundles.
\begin{enumerate}
\item \label{prop:2.4.1}
For some small collar neighborhood $N(\partial W_f) \subset W_f$, the composition of the map $q_f {\mid}_{{q_f}^{-1}(N(\partial W_f))}$ onto $N(\partial W_f)$ with the canonical projection to $\partial W_f$ is the projection of a linear bundle whose fiber is the {\rm(}$m-n+1${\rm )}-dimensional unit disk. 
\item \label{prop:2.4.2}
The restriction of $q_f$ to the preimage of $W_f-{\rm Int}\ N(\partial W_f)$ is the projection of a smooth bundle over $W_f-{\rm Int}\ N(\partial W_f)$ whose fiber is an {\rm (}$m-n${\rm )}-dimensional standard sphere.
\end{enumerate}
\end{enumerate}
\end{Prop}

This is explicitly presented in \cite{saeki} as a proposition in the case $m>n \geq 1$. 
The theory of special generic maps from $m$-dimensional closed manifolds into $n$-dimensional manifolds with no boundaries with $m=n \geq 1$ is essentially the theory of Eliashberg (\cite{eliashberg}).

Let $m \geq n \geq 1$ be integers again. Conversely, if we have a smooth immersion $\bar{f}:W_f \rightarrow N$ of an $n$-dimensional compact smooth manifold $W_f$ into an $n$-dimensional manifold $N$ with no boundary, we have a special generic map $f_0:M_0 \rightarrow N$ on a suitable $m$-dimensional closed manifold $M_0$ into $N$ enjoying the following properties.
\begin{enumerate}
\item There exists an $n$-dimensional compact manifold $W_{f_0}$ and a smooth immersion $\bar{f_0}:W_{f_0} \rightarrow N$.
\item There exists a diffeomorphism ${\phi}:W_{f} \rightarrow W_{f_0}$ satisfying $\bar{f}=\bar{f_0} \circ \phi$.
\item There exists a smooth surjection $q_{f_0}:M \rightarrow W_{f_0}$ and $f_0=\bar{f_0} \circ q_{f_0}$.
\item $q_{f_0}$ maps the singular set $S(f_0)$ of $f_0$ onto the boundary $\partial W_{f_0} \subset W_{f_0}$ as a diffeomorphism.
\item We have the following two bundles.
\begin{enumerate}
\item For some small collar neighborhood $N(\partial W_{f_0}) \subset W_{f_0}$, the composition of the map $q_{f_0} {\mid}_{{q_{f_0}}^{-1}(N(\partial W_{f_0}))}$ onto $N(\partial W_{f_0})$ with the canonical projection to $\partial W_{f_0}$ is the projection of a trivial linear bundle whose fiber is the {(\rm }$m-n+1${\rm )}-dimensional unit disk. 
\item The restriction of $q_{f_0}$ to the preimage of $W_{f_0}-{\rm Int}\ N(\partial W_{f_0})$ is the projection of a trivial smooth bundle over $W_{f_0}-{\rm Int}\ N(\partial W_{f_0})$ whose fiber is an {\rm (}$m-n${\rm )}-dimensional standard sphere.
\end{enumerate}
\end{enumerate}

This property is introduced as propositions in most of articles in REFERENCES by the author. This is also regarded as a fundamental exercise.

Hereafter, for a finite set $X$, $|X|$ denotes the size of $X$.

\begin{Ex}
\label{ex:1}
Let $m \geq n \geq 2$ be integers. Let $\{S^{k_j} \times S^{m-k_j}\}_{j \in J}$ be a family of finitely many products of two unit spheres where $k_j$ is an integer satisfying $1 \leq k_j \leq n-1$. We consider a connected sum of these $|J|$ manifolds in the smooth category. Let $M_0$ be an $m$-dimensional closed and connected manifold diffeomorphic to the obtained manifold. We have a special generic map $f_0$ so that $W_{f_0}$ is diffeomorphic to a manifold represented as a boundary connected sum of $|J|$ manifolds in $\{S^{k_j} \times D^{n-k_j}\}_{j \in J}$. Of course the boundary connected sum is considered in the smooth category. Furthermore, we can take $N={\mathbb{R}}^n$ and $f_0 {\mid}_{S(f_0)}$ as an embedding.

\end{Ex}
\begin{Prop}[E. g. \cite{saeki}]
\label{prop:3}
Let $m>n \geq 1$ be integers. 
For a special generic map $f:M \rightarrow N$ on an $m$-dimensional closed and connected manifold $M$ into an $n$-dimensional manifold $N$ with no boundary, we have the following properties.
\begin{enumerate}
	\item We have an {\rm (}$m+1${\rm )}-dimensional compact and connected topological manifold {\rm (PL)} manifold $W$ whose boundary {\rm (}resp. considered in the PL category{\rm )} is $M$ and which collapses to $W_f$ where $W_f$ is an $n$-dimensional compact and smooth manifold "$W_f$ in Proposition \ref{prop:2}" and identified with a suitable CW subcomplex {\rm (}resp. subpolyhedron{\rm )} of $W$.
	\item For the canonical inclusion $i_M:M \rightarrow W$ and a continuous {\rm (}resp. PL{\rm )} map $r_f:W \rightarrow W_f$ giving a collapsing to $W_f$, $q_f=r \circ i_M$ holds.
	\item If $M$ is oriented, then $W$ can be chosen as an oriented manifold. 
    \item If $m-n=1,2,3$ in addition, then $W$ can be chosen as a smooth manifold and $r$ as a smooth map.
    \end{enumerate}
\end{Prop}
Let $m > n \geq 1$ be integers again. Conversely, Suppose that we have a smooth immersion $\bar{f}:W_f \rightarrow N$ of an $n$-dimensional compact and connected smooth manifold $W_f$ into an $n$-dimensional manifold $N$ with no boundary. We have a suitable special generic map $f_0:M_0 \rightarrow N$ on a suitable $m$-dimensional closed and connected manifold $M_0$ into $N$ satisfying the properties just after Proposition \ref{prop:2} and we can have an {\rm (}$m+1${\rm )}-dimensional compact and connected smooth manifold $W$ and a smooth map $r:W \rightarrow W_f$ as in Proposition \ref{prop:3}.
For more general propositions of this type, see \cite{saekisuzuoka} and see papers \cite{kitazawa0.1,kitazawa0.2,kitazawa0.3} by the author.

Note also that we concentrate on the case $m-n=1,2,3$ in main ingredients and we discuss problems in the smooth category in situations of Proposition \ref{prop:3} in main ingredients of our paper.

The following two theorems give some characterizations of manifolds admitting special generic maps where the classes of the manifolds and the dimensions of the Euclidean spaces of the targets are fixed.

\begin{Thm}
\label{thm:1}
Let $m$ be a positive integer. We have the following characterizations where connected sums are considered in the smooth category.
\begin{enumerate}
\item {\rm (}Reeb's theorem{\rm )}
\label{thm:1.0}
An $m$-dimensional closed and connected manifold $M$ admits a special generic map into $\mathbb{R}$ if and only if either of the following two holds.
\begin{enumerate}
	\item $m \neq 4$ and $M$ is a homotopy sphere.
	\item $m=4$ and $M$ is a standard sphere.
\end{enumerate}
\item {\rm (\cite{saeki})}
\label{thm:1.1}
For $m \geq 2$, an $m$-dimensional closed and connected manifold $M$ admits a special generic map into ${\mathbb{R}}^2$ if and only if either of the following two holds.
\begin{enumerate}
\item $M$ is a homotopy sphere which is not a $4$-dimensional exotic sphere.
\item $M$ is diffeomorphic to a manifold represented as a connected sum of the total spaces of smooth bundles over $S^1$ whose fibers are diffeomorphic to homotopy spheres which are not $4$-dimensional exotic spheres.
\end{enumerate}
\item {\rm (\cite{saeki})}
\label{thm:1.2}
For $m=4,5$, an $m$-dimensional closed and simply-connected manifold $M$ admits a special generic map into ${\mathbb{R}}^3$ if and only if either of the following two holds.
\begin{enumerate}
\item $M$ is a standard sphere.
\item $M$ is diffeomorphic to a manifold represented as a connected sum of the total spaces of linear bundles over $S^2$ whose fibers are to the {\rm (}$m-2${\rm )}-dimensional unit sphere.
\end{enumerate}
\item {\rm (\cite{nishioka})}
\label{thm:1.3}
For $m=5$, an $m$-dimensional closed and simply-connected manifold $M$ admits a special generic map into ${\mathbb{R}}^4$ if and only if either of the following two holds.
\begin{enumerate}
\item $M$ is a standard sphere.
\item $M$ is diffeomorphic to a manifold represented as a connected sum of the total spaces of linear bundles over $S^2$ whose fibers are the {\rm (}$m-2${\rm )}-dimensional unit sphere.
\end{enumerate}
\end{enumerate}
\end{Thm}

We introduce notation and terminologies on {\it homology groups}, {\it cohomology groups} and {\it rings} and {\it homotopy groups}. For elementary or advanced theory, consult \cite{hatcher} for example.

Let $(X,X^{\prime})$ be a pair of topological spaces satisfying $X^{\prime} \subset X$ where $X^{\prime}$ can be the empty set. 
Let $A$ be a commutative ring. The {\it homology group} ({\it cohomology group}) of the pair $(X,X^{\prime})$ of topological spaces satisfying $X^{\prime} \subset X$ is defined and denoted by $H_{\ast}(X,X^{\prime};A)$ (resp. $H^{\ast}(X,X^{\prime};A)$) where the {\it coefficient ring} is $A$. If $A$ is isomorphic to the ring of all integers, denoted by $\mathbb{Z} \subset \mathbb{R}$ (resp. the ring of all rational numbers, denoted by $\mathbb{Q} \subset \mathbb{R}$), then the homology group and the cohomology group are called the {\it integral} (resp. {\it rational}) {\it homology group} and the {\it integral} (resp. {\it rational}) {\it cohomology group}, respectively. The {\it $k$-th homology group} ({\it cohomology group}) is denoted by $H_k(X,X^{\prime};A)$ (resp. $H^k(X,X^{\prime};A)$). If $A$ is isomorphic to $\mathbb{Z}$ (resp. $\mathbb{Q}$), then we add "{\it integral}" (resp. "{\it rational}") after "$k$-th" as before. If $X^{\prime}$ is empty, then we may omit ",$X^{\prime}$" in the notation and the homology group (cohomology group) of the pair $(X,X^{\prime})$ is also called the {\it homology group} (resp. {\it cohomology group}) of $X$. We can define similarly for $k$-th homology groups and cohomology groups and integral and rational ones.
{\rm (}{\it Co}{\rm )}{\it homology classes} of $(X,X^{\prime})$ (or $X$) are elements of the (resp. co)homology groups. If the degree is $k$ for each element, then we add {\it "$k$-th"} as for the homology group and cohomology group. We can define similarly for integral or rational homology groups and cohomology groups.

The {\it $k$-th homotopy group} of a topological space $X$ is denoted by ${\pi}_k(X)$ where we can define the group.

Let $(X_1,{X_1}^{\prime})$ and $(X_2,{X_2}^{\prime})$ be pairs of topological spaces satisfying ${X_1}^{\prime} \subset X_1$ and ${X_2}^{\prime} \subset X_2$ where the second topological spaces of the pairs can be empty. For a continuous map $c:X_1 \rightarrow X_2$ satisfying $c({X_1}^{\prime}) \subset {X_2}^{\prime}$, $c_{\ast}:H_{\ast}(X_1,{X_1}^{\prime};A) \rightarrow H_{\ast}({X_2},{X_2}^{\prime};A)$ and $c^{\ast}:H^{\ast}({X_2},{X_2}^{\prime};A) \rightarrow H^{\ast}(X_1,{X_1}^{\prime};A)$ denote canonically induced homomorphisms. For a continuous map $c:X_1 \rightarrow X_2$, $c_{\ast}:{\pi}_k(X_1) \rightarrow {\pi}_k(X_2)$ also denotes the induced homomorphism between the homotopy groups of degree $k$. 

The cup products for an ordered pair $(c_1,c_2) \in H^{\ast}(X;A) \times H^{\ast}(X;A)$ and a sequence $\{c_j\}_{j=1}^l \subset H^{\ast}(X;A)$ of $l>0$ cohomology classes are important. $c_1 \cup c_2$ and ${\cup}_{j=1}^l c_j$ denote them. This gives $H^{\ast}(X;A)$ the structure of a graded commutative algebra over $A$ and this is the cohomology ring of $X$ whose coefficient ring is $A$.
\begin{Thm}[\cite{kitazawa}]
\label{thm:2}
Let $m>n \geq 1$ be integers. Let $l>0$ be another integer.
Let $M$ be an $m$-dimensional closed and connected manifold.
Let $A$ be commutative ring. Let there exist a sequence $\{a_j\}_{j=1}^l \subset H^{\ast}(M;A)$ satisfying the following three. 
\begin{itemize}
\item The cup product ${\cup}_{j=1}^l a_j$ is not the zero element.
\item The degree of each cohomology class in $\{a_j\}_{j=1}^l$ is smaller than or equal to $m-n$.
\item The sum of the degrees of the $l$ cohomology classes in $\{a_j\}_{j=1}^l$ is greater than or equal to $n$.
\end{itemize}
Then $M$ does not admit special generic maps into any $n$-dimensional connected manifold which is non-closed and which has no boundary.
\end{Thm}
We review a proof without exposition on a {\it handle} and its {\it index} for smooth manifolds and polyhedra including PL manifolds.
\begin{proof}
Suppose that $M$ admits a special generic map into an $n$-dimensional connected manifold $N$ which is not compact and which has no boundary. We can take an ($m+1$)-dimensional compact and connected PL manifold $W$ as in Proposition \ref{prop:3} (note that in the main ingredients we consider the case $m-n=1,2,3$ only and discuss problems in the smooth category). 
$W_f$ is a smooth compact and connected manifold smoothly immersed into the connected and non-compact manifold $N$ with no boundary. As a result it is (simple) homotopy equivalent to an ($n-1$)-dimensional compact and connected polyhedron. 
$W$ is (simple) homotopy equivalent to $W_f$. $W$ is shown to be a PL manifold obtained by attaching handles to $M \times \{0\} \subset M \times [-1,0]$ whose indices are greater than $(m+1)-{\dim W_f}=m-n+1$ where the boundary $\partial W=M$ can be identified with $M \times \{-1\} \subset M \times [-1,0]$.
We can take a unique cohomology class $b_j \in H^{\ast}(W;A)$ satisfying $a_j={i_M}^{\ast}(b_j)$ where $i_M$ is as in Proposition \ref{prop:3}. $W$ has the (simple) homotopy type of an ($n-1$)-dimensional polyhedron. This means that the cup product ${\cup}_{j=1}^l a_j$ is the zero element, which contradicts the assumption. This completes the proof.
\end{proof}
\begin{Cor}[\cite{kitazawa2}]
\label{cor:1}
Let $m \geq n \geq 1$ be integers. Let $N$ be an $n$-dimensional connected manifold which is not compact and which has no boundary. We have the following two.
\begin{enumerate}
\item
\label{cor:1.1}
 Suppose that $m>n$ in addition. Then the $m$-dimensional real projective space does not admit special generic maps into $N$. 
\item
\label{cor:1.2}
 Suppose also that $n<m-1$ and that $m$ is an even integer. The $\frac{m}{2}$-dimensional complex projective space, which is also an $m$-dimensional smooth, closed and simply-connected manifold, does not admit special generic maps into $N$.
\end{enumerate}
\end{Cor}
In Corollary \ref{cor:1} in the case $m=n$, \cite{eliashberg, kikuchisaeki} produce useful tools for example. Our exposition on Corollary \ref{cor:2} (Main Theorem) and Remarks \ref{rem:2} and \ref{rem:3} refer to some of the theory. 

\begin{Rem}
\label{rem:2}
A slide \cite{wrazidlo2} of a related talk in a conference shows a proof of Corollary \ref{cor:1} for the $7$-dimensional real projective space. This is based on the theory on $7$-dimensional homotopy spheres and special generic maps on them, presented shortly in the first section, before the appearance of \cite{kitazawa2}. After \cite{kitazawa2} appeared, \cite{wrazidlo3} announced another proof. It investigates restrictions on the torsion group of the integral homology group for a closed and connected smooth manifold whose rational homology group is isomorphic to that of a sphere. Note also that the dimension $7$ is the smallest dimension where we have discovered exotic spheres. \cite{milnor} is a pioneering work, followed by \cite{eellskuiper} and \cite{kervairemilnor} for example.
\end{Rem}

For other related studies on special generic maps, see \cite{burletderham, furuyaporto, saekisakuma, saekisakuma2, sakuma} for example.
\section{The $k$-dimensional complex projective space ${\mathbb{C}}P^k$.}
${\mathbb{C}P}^k$ denotes the $k$-dimensional complex projective space for each integer $k>0$.

Hereafter, we need notions of the {\it $j$-th Stiefel-Whitney class} and the {\it $j$-th Pontrjagin class} of a smooth manifold $M$. They are uniquely defined elements of $H^j(M;\mathbb{Z}/2\mathbb{Z})$ where $\mathbb{Z}/2\mathbb{Z}$ is a finite field of order $2$, and $H^{4j}(M;\mathbb{Z})$, respectively. More precisely, they are the {\it $j$-th Stiefel-Whitney class} and the {\it $j$-th Pontrjagin class} of the tangent bundle of $M$, respectively. For systematic and general expositions on such cohomology classes or characteristic classes for linear bundles, see \cite{milnorstasheff} for example.

A {\it spin} manifold is an orientable smooth manifold whose second Stiefel-Whitney class is zero. Such linear bundles are said to be {\it spin} in general.

\begin{Thm}
	\label{thm:3}
	${\mathbb{C}P}^k$ is a closed, connected and smooth manifold of dimension $2k$ and a complex manifold of dimension $k>0$ enjoying the following properties.
	\begin{enumerate}
		\item ${\mathbb{C}P}^k$ is simply-connected and in the case $k=1$ it is diffeomorphic to the $2$-dimensional unit sphere $S^2$.
		\item $H_j({\mathbb{C}P}^k;\mathbb{Z})$ is isomorphic to the trivial group for any odd integer $0 \leq j \leq 2k$ and isomorphic to $\mathbb{Z}$ for any even integer $0\leq j \leq 2k$.
		\item $H^{2i}({\mathbb{C}P}^k;\mathbb{Z})$ is generated by the cup product ${u_0}^i:={\cup}_{j=1}^i u_0$ for an arbitrary integer $1 \leq i \leq k$ and a generator $u_0 \in H^2({\mathbb{C}P}^k;\mathbb{Z})$.
		\item If $k$ is odd, then ${\mathbb{C}P}^k$ is spin.
		\item The $j$-th Pontrjagin class of ${\mathbb{C}P}^k$ is not the zero element for each integer $1 \leq j \leq \frac{k}{2}$

	\end{enumerate}
\end{Thm}
\section{Proofs of our Main Theorems.}
In proving our Main Theorem \ref{mthm:1}, the following theorem is a main ingredient.
\begin{Thm}
	\label{thm:4}
	${{\mathbb{C}P}^3}$ does not admit special generic maps into ${\mathbb{R}}^5$.
\end{Thm}
 \begin{Cor}[Main Theorem \ref{mthm:1}]
	\label{cor:2}
	${{\mathbb{C}P}^3}$ does not admit special generic maps into any Euclidean space. 
\end{Cor}
\begin{proof}
	Theorem \ref{thm:2} or Corollary \ref{cor:1} (\ref{cor:1.2}) yields the fact that the ${{\mathbb{C}P}^3}$ does not admit special generic maps into ${\mathbb{R}}^n$ for $n=1,2,3,4$. This does not admit ones into ${\mathbb{R}}^6$ by \cite{eliashberg}. This says that a closed and orientable manifold admits a special generic map into the Euclidean space of the same dimension if and only if the so-called {\it Whitney sum} of the tangent bundle and a trivial linear bundle whose fiber is $\mathbb{R}$ is trivial. Theorem \ref{thm:4} completes the proof. 
\end{proof}

Before we prove Theorem \ref{thm:4}, we shortly refer to several notions.

The {\it fundamental class} of a compact, connected and oriented smooth, PL or topological manifold $Y$ is the canonically defined ($\dim Y$)-th homology class. This is defined as the generator of the group $H_{\dim Y}(Y,\partial Y;A)$, isomorphic to $A$, and compatible with the orientation where $A$ is a commutative ring having the identity element different from the zero element. 
Let $i_{Y,X}:Y \rightarrow X$ be a smooth, piecewise smooth, or topologically flat embedding satisfying $i_{Y,X}(\partial Y) \subset \partial X$ and $i_{Y,X}({\rm Int}\ Y) \subset {\rm Int}\ X$. In other words, $Y$ is embedded  {\it properly} in $X$.
Let $h$ be a homology class in $H_{\ast}(X,\partial X;A)$. If the value of the homomorphism ${i_{Y,X}}_{\ast}$ induced by the smooth, piecewise smooth, or topologically flat embedding $i_{Y,X}:Y \rightarrow X$ at the fundamental class of $Y$ is $h$, then $h$ is said to be {\it represented} by the orientable (or oriented) submanifold $Y$. 

We do not need orientations in the case where $A$ is $\mathbb{Z}/2\mathbb{Z}$, the field of order $2$, or more generally, the commutative ring consisting of elements whose orders are at most $2$.

The {\it Poincar\'e dual} ${\rm PD}(h_{\rm c}) \in H^{\dim Y-j}(Y,\partial Y;A)$ to a homology class $h_{\rm c} \in H_{j}(Y,\partial Y;A)$, the {\it Poincar\'e dual} ${\rm PD}(c_{\rm c}) \in H_{\dim Y-j}(Y,\partial Y;A)$ to a cohomology class $c_{\rm c} \in H^{j}(Y,\partial Y;A)$ and Poincar\'e duality theorem for the manifold $Y$ are also fundamental and important.

The {\it cohomology dual} to a homology class of a basis of a uniquely defined maximal free subgroup of a homology group of $(X,X^{\prime} \subset X)$ {\it respecting the basis} is also important for topological spaces $X \supset X^{\prime}$ where $X^{\prime}$ may be empty. It is a uniquely defined cohomology class in the cohomology group of the pair $(X,X^{\prime})$ whose degree is same as that of the homology class.

For notions and facts here, consult \cite{hatcher} again for example.

We show Theorem \ref{thm:4} by several methods.
\subsection{The first method.}
\begin{Thm}
	\label{thm:5}
	Let $m>5$ be an integer. An $m$-dimensional closed and simply-connected manifold whose 2nd integral cohomology group is isomorphic to $\mathbb{Z}$ and which has a generator $u \in H^2(M;\mathbb{Z})$ whose square $u \cup u$ is not divisible by $2$ does not admit special generic maps into ${\mathbb{R}}^5$.

\end{Thm}
We need some additional notions from characteristic classes.
{\it Spin} bundles are real vector bundles or linear bundles which are {\it orientable} and whose 2nd Stiefel-Whitney classes are zero where {\it orientable} linear bundles mean bundles whose structure groups are reduced to rotation groups, or groups consisting of linear transformations preserving the orientations of the fibers.

We also need fundamental or advanced arguments on singularity theory of differentiable maps and theory of PL or smooth manifolds. For example we encounter {\it generic} immersions, embeddings and maps. We omit precise expositions in the present paper. See \cite{golubitskyguillemin} for related theory for example.

\begin{proof}[A proof of Theorem \ref{thm:5}]
	For the proof of the present theorem, only the assumption $m=6$ is essentially new. This reviews several main theorems and their proofs of \cite{kitazawa3}. Our exposition here may be a bit different from ones there.
	
	Suppose that a special generic map $f:M \rightarrow {\mathbb{R}}^5$ exists. We investigate the topology of $W_f$ in Proposition \ref{prop:2}. This is a $5$-dimensional compact and connected manifold smoothly immersed into ${\mathbb{R}}^5$ via $\bar{f}:W_f \rightarrow {\mathbb{R}}^5$.
	Proposition \ref{prop:3} and the proof of Theorem \ref{thm:2} yield the triviality of ${\pi}_1(W_f)$. 
	More precisely, $W$ in the proof of Theorem \ref{thm:2} is obtained by attaching handles whose indices are greater than $m-5+1=m-4 \geq 2$ to $M \times \{0\} \subset M \times [-1,0]$. Here the boundary $M$ is identified with $M \times \{-1\}$ as the proof of Theorem \ref{thm:2}.
	See also related arguments in \cite{saekisuzuoka} or Corollary 4.8 there. See also articles \cite{kitazawa0.1,kitazawa0.2,kitazawa0.3} by the author.
	We may regard that homology classes in our proof are represented by compact and oriented smooth submanifolds thanks to the celebrating theory of \cite{thom}. \\ 
	
	Suppose that the rank of $H_2(W_f;\mathbb{Z})$ is greater than $1$. We have two 2nd integral homology classes $e_1,e_2 \in H_2(W_f;\mathbb{Z})$ satisfying the following conditions.
	\begin{itemize}
		\item $e_i$ is not divisible by any integer greater than $1$ for $i=1,2$.
		\item $e_1$ and $e_2$ are mutually independent and of infinite order.
	\end{itemize} 
	$W_f$ is simply-connected. These classes are represented by oriented smooth submanifolds diffeomorphic to $S^2$ embedded in ${\rm Int}\ W_f$. 
	Let $e_i:S^2 \rightarrow W_f$ also denote the smooth embedding giving the oriented submanifold for the homology class $e_i$.
	If we restrict the map $q_f:M \rightarrow W_f$ to the preimages of the submanifolds, then we have the projections of linear bundles whose fibers are circles. They do not admit sections in general. We consider a smooth homotopy $E_i:S^2 \times [0,1] \rightarrow W_f$ from the original embedding $e_i$ of the $2$-dimensional standard sphere into ${\rm Int}\ W_f \subset W_f$ to another smooth embedding enjoying the following properties. We have this due to the structure of the special generic map $f$ and the smooth surjection $q_f:M \rightarrow W_f$.
	\begin{itemize}
		\item There exists a finite subset $J_i \subset (0,1)$.
		\item The restriction of $E_i$ to each subset of the form $S^2 \times \{p\} \subset S^2 \times [0,1]$ is a smooth embedding the intersection of the boundary $\partial W_f$ and whose image is a finite set.
		\item The restriction of $E_i$ to each connected component of $S^2 \times ([0,1]-J_i)$ is regarded as a smooth isotopy.
		\item There exists a smooth embedding $s_i:S^2 \rightarrow M$ satisfying $q_f \circ s_i={e_i}^{\prime}$ where ${e_i}^{\prime}:S^2 \rightarrow W_f$ is defined by ${e_i}^{\prime}(x):=E_i(x,1)$.  
	\end{itemize} 
	$s_i$ is regarded as a variant of the section for the original bundle. 
	We can see that two mutually independent 2nd integral homology classes are represented by the oriented smooth submanifolds given by the smooth embeddings $s_1$ and $s_2$. They are of infinite order. This contradicts the assumption on the rank of $H^2(M;\mathbb{Z})$. 

	Suppose that the rank of $H_2(W_f;\mathbb{Z})$ is $1$. We have a similar variant $s:S^2 \rightarrow M$ of a section of the bundle over a suitable smoothly embedded $2$-dimensional standard sphere in ${\rm Int}\ W_f$. We have a 2nd integral homology class $e \in H_2(W_f;\mathbb{Z})$ satisfying the following conditions and represented by the sphere in ${\rm Int}\ W_f \subset W_f$. 
	\begin{itemize}
		\item $e$ is not divisible by any integer greater than $1$.
		\item $e$ is of infinite order.
	\end{itemize} 
	We can define the cohomology dual $e^{\ast} \in H^2(W_f;\mathbb{Z})$ to $e$ respecting the basis. 
	By Poincar\'e duality theorem, we can take a $3$-dimensional, compact, and connected oriented submanifold $Y$ smoothly embedded in $W_f$ suitably in such a way that the Poincar\'e dual to $e^{\ast}$ is represented by this. We can take $Y$ enjoying the following properties.
	\begin{itemize}
		\item The boundary $\partial Y$ is embedded in the boundary $\partial W_f$ and the interior ${\rm Int}\ Y$ is embedded in the interior ${\rm Int}\ W_f$. In other words, $Y$ is embedded properly in $W_f$.
		\item ${q_f}^{-1}(Y)$ is regarded as an ($m-2$)-dimensional closed and connected manifold and regarded as one of the domain of a special generic map whose image is $Y$ and the manifold of whose target is a $3$-dimensional, non-compact, connected and orientable manifold with no boundary. 
	\end{itemize}
	$W_f$ is simply-connected and this with the structure of the special generic map yields the fact that ${q_f}^{-1}(Y)$ is orientable. We can also see this from the fact that $M$, $W_f$ and $Y$ are orientable for example.
	By arguments on intersection theory and Poincar\'e duality, there exists a suitable integral homology class $e_s \in H_2(M;\mathbb{Z})$ of infinite order represented by an oriented submanifold given by $s:S^2 \rightarrow M$ and the class is not divisible by any integer greater than $1$. The Poincar\'e dual to its cohomology dual respecting the basis, which is defined as an element of $H^2(M;\mathbb{Z})$ uniquely, is represented by a suitably oriented submanifold ${q_f}^{-1}(Y)$. We consider so-called {\it generic} {\it self-intersections} of $Y$ and ${q_f}^{-1}(Y)$. $Y$ is a spin manifold since it is $3$-dimensional, compact and orientable. $W_f$ is smoothly immersed into ${\mathbb{R}}^5$ and thus spin. By arguments on spin bundles and manifolds, a generic self-intersection of $Y$ can be regarded as the union of the empty sets or $1$-dimensional smoothly and properly embedded submanifolds chosen for all homology classes of $H_1(Y,\partial Y;\mathbb{Z}/2\mathbb{Z})$. This satisfies the following conditions.
	\begin{itemize}
		\item For each homology class of $H_1(Y,\partial Y;\mathbb{Z}/2\mathbb{Z})$, it is the empty set or represented by each connected component of the chosen submanifold. 
		\item For each homology class of $H_1(Y,\partial Y;\mathbb{Z}/2\mathbb{Z})$, the chosen submanifold is either a disjoint union consisting of only $1$-dimensional compact and connected manifolds diffeomorphic to a closed interval or one consisting of only circles, unless it is empty.
		\item For each homology class of $H_1(Y,\partial Y;\mathbb{Z}/2\mathbb{Z})$, the chosen submanifold is a disjoint union of $1$-dimensional compact and connected manifolds of an even number, unless it is empty.
	\end{itemize}
	Note that if the coefficient ring for the homology groups is $\mathbb{Z}/2\mathbb{Z}$, then we do not need orientations for compact and connected submanifolds.
	The structure of special generic maps yields the following arguments where the coefficient ring for the homology groups is $\mathbb{Z}/2\mathbb{Z}$. We can argue as follows.
	\begin{itemize}
		\item A generic self-intersection of ${q_f}^{-1}(Y)$ is the union of manifolds diffeomorphic to $S^1 \times S^{m-5}$ of an even number and ($m-4$)-dimensional homotopy spheres of an even number, unless it is empty. $W_f$ is simply-connected and the preimage of each circle in ${\rm Int}\ W_f \subset W_f$ is the total space of a trivial smooth bundle. The ($m-4$)-dimensional homotopy spheres are the preiamges of the connected components diffeomorphic to $1$-dimensional compact and connected manifolds in $W_f$ which are smoothly and properly embedded: the interiors are embedded in ${\rm Int}\ W_f$ and the boundaries are in $\partial W_f$.
		\item For each of the ($m-4$)-dimensional closed and connected manifolds before, consider the value of the homomorphism induced by the inclusion into $M$ at the fundamental class of the compact and connected suitably oriented manifold of the domain. Take the sum for all of them. Then we have the zero element $0 \in H_{m-4}(M;\mathbb{Z}/2\mathbb{Z})$. 
	\end{itemize}
	
	For arguments here, see also \cite{kitazawa3} for example.
	We can take the cohomology dual ${e_s}^{\ast}$ to $e_s$ respecting the basis.  
	This yields that the cup product ${e_s}^{\ast} \cup {e_s}^{\ast}$ is divisible by $2$. This is a contradiction. 
	
	The rank of $H_2(W_f;\mathbb{Z})$ is thus $0$. The given cohomology class $u$ is the cohomology dual to some suitable 2nd integral homology class $v_u$ respecting the basis. $v_u$ is represented by a smooth embedding $e_0:S^2 \rightarrow M$. Assume that $q_f \circ e_0(S^2) \subset {\rm Int}\ W_f$. $\partial W_f$ is a $4$-dimensional, closed and orientable smooth manifold. If $m \geq 7$, then we can take a suitable embedding $e_0$ satisfying this assumption by the condition on the dimensions of the manifolds.
	$q_f \circ e_0$ can be regarded as a null-homotopic map as a map into ${\rm Int}\ W_f$ if $H_2(W_f;\mathbb{Z})$ is the trivial group. $H_2(W_f;\mathbb{Z})$ may be not free and in this case it is a finite commutative group which is not the trivial group. In such a case, by considering finitely many smooth embeddings smoothly isotopic to $e_0$ and a suitable connected sum of these copies, we have a desired situation. This is also regarded as reviewing some ingredients of Main Theorem 2 of \cite{kitazawa3} and its proof.
	
	The fact that $W_f$ is simply-connected plays an important role here. This yields the fact that $e_0:S^2 \rightarrow M$ or a suitable connected sum of finitely many smooth embeddings smoothly isotopic to $e_0$ is (smoothly) null-homotopic. This is a contradiction. 
	
	This completes the proof for the case $m \geq 7$.
	
	Hereafter assume that $m=6$.
	
	We can do so that $q_f \circ e_0(S^2) \bigcap \partial W_f$ is a finite set. 
If it is an empty set, then $v_u$ is represented by a smooth embedding $e_0:S^2 \rightarrow M$ and shown to be the zero element, which is a contradiction. This is also due to the fact that $W_f$ is simply-connected as before and also reviews the proof of Main Theorem 2 of \cite{kitazawa3}. The ranks of the groups $H_2(M;\mathbb{Z})$ and $H^4(M;\mathbb{Z})$ are $1$ by Poincar\'e duality for $M$. Some element represented as the $k_u$ times the Poincar\'e dual to $u$ is shown to be represented by ${q_f}^{-1}(\partial W_f)$ for a suitable non-zero integer $k_u \neq 0$. If $k_u=0$, then by a fundamental argument on the intersection, we can make the finite set $q_f \circ e_0(S^2) \bigcap \partial W_f$ empty and this is also a contradiction. 
By Proposition \ref{prop:2} with the facts that $M$ and $W_f$ are simply-connected and that $H^2(W_f;\mathbb{Z})$ is the trivial group, two bundles in Proposition \ref{prop:2} (\ref{prop:2.4}) are trivial. Here a well-known fact on classifications of linear bundles whose fibers are circles and whose structure groups consist of orientation preserving linear transformations, or the fact that such bundles over a fixed base space is classified by the 2nd integral cohomology group of the base space is a key. For this, consult \cite{milnorstasheff,steenrod} for example. 

We return to the fact that the linear bundle whose fiber is the $2$-dimensional unit disk in Proposition \ref{prop:2} (\ref{prop:2.4}) is trivial. This means that the self-intersection of $\partial W_f$ can be empty. 

Applying Poincar\'e duality or intersection theory for $M$ again, the cup product $u \cup u$ is shown to be the zero element of $H^4(M;\mathbb{Z})$. This is a contradiction.
	
	This completes the proof.

\end{proof}

Theorem \ref{thm:4} is obtained as a specific case of Theorem \ref{thm:5}. See also Theorem \ref{thm:3} again.

\subsection{The second method.}
The following theorem is a main theorem of \cite{kitazawa3} or Main Theorem 1 there if we replace $m \geq 6$ by $m \geq 7$. We also review the proof of this case here.
\begin{Thm}
	\label{thm:6}
	For a closed and simply-connected manifold $M$ of dimension $m \geq 6$ admitting a special generic map $f:M \rightarrow {\mathbb{R}}^5$ whose singular set $S(f)$ is connected, the cup product $c_1 \cup c_2$ is the zero element of $H^4(M;\mathbb{Z})$ for any pair $c_1,c_2 \in H^2(M;\mathbb{Z})$ of integral cohomology classes.
\end{Thm}
\begin{proof}
	We apply some arguments on Proposition 3.10 of \cite{saeki}.
	This applies an exact sequence essentially regarded as the following  homology exact sequence for the pair $(W,M)$
	$$\rightarrow H_3(W,M;\mathbb{Z}) \rightarrow H_2(M;\mathbb{Z}) \rightarrow H_2(W;\mathbb{Z}) \rightarrow H_2(W,M;\mathbb{Z}) \rightarrow$$
	where we abuse $W$ in Proposition \ref{prop:3}. 
	
	Let $m=6$ for a while.
	$H_3(W,M;\mathbb{Z})$ is isomorphic to $H^{4}(W;\mathbb{Z})$,  $H^4(W_f;\mathbb{Z})$ and $H_1(W_f,\partial W_f;\mathbb{Z})$ by Poincar\'e duality theorem for $W$ and $W_f$ where we abuse $W_f$ similarly. $H_1(W_f;\mathbb{Z})$ is the trivial group since $W_f$ is simply-connected as presented in the proof of Theorem \ref{thm:5} for example. $H_1(W_f,\partial W_f;\mathbb{Z})$ is the trivial group from the assumption that $\partial W_f=q_f(S(f))$ is connected and the homology exact sequence for the pair $(W_f,\partial W_f)$. $H_2(W,M;\mathbb{Z})$ is isomorphic to $H^{5}(W;\mathbb{Z})$ by Poincar\'e duality theorem and it is isomorphic to $H^{5}(W_f;\mathbb{Z})$ and the trivial group since $W$ collapses to $W_f$, which has the homotopy type of a polyhedron whose dimension is smaller than $5$, by Proposition \ref{prop:3}. Thus ${q_f}_{\ast}:H_2(M;\mathbb{Z}) \rightarrow H_2(W_f;\mathbb{Z})$ is an isomorphism from Proposition \ref{prop:3}, saying that $q_f$ is represented as the composition of the inclusion $i_M:M \rightarrow W$ with the map $r_f:W \rightarrow W_f$ giving a collapsing.
	 
	Note that we already have this isomorphism for $m \geq 7$. We have this as presented in the beginning of the proof of Theorem \ref{thm:5}. The difference of the dimensions of the manifolds of the domains and the targets are $7-5=2$ and this is a key. $W$, collapsing to $W_f$, is obtained by attaching handles whose indices are greater than $7-5+1$=$3$ to $M \times \{0\} \subset M \times [-1,0]$ where $M$ or the boundary is identified with $M \times \{-1\}$.
	
	 Let $l \geq 0$  denote the rank of $H_2(W_f;\mathbb{Z})$. The case $l=0$ is trivial and we consider the case $l>0$. We can obtain a family $\{Y_j\}_{j=1}^l$ of $l$ 3-dimensional compact, connected and orientable manifolds smoothly embedded in $W_f$ in a way similar to that of $Y$ in the proof of Theorem \ref{thm:5}. We have a suitable 2nd integral homology classes forming a basis of a free subgroup of rank $l$ of $H_2(W_f;\mathbb{Z})$ in such a way that the integral homology classes represented by these $3$-dimensional manifolds are regarded as the Poincar\'e duals to the cohomology duals to these 2nd integral homology classes in the basis of the free subgroup of $H_2(W_f;\mathbb{Z})$ respecting the basis.
	We have the family $\{{q_f}^{-1}(Y_j)\}$ of ($m-2$)-dimensional closed, connected and orientable manifolds. The integral homology classes represented by these manifolds with suitable orientations are also regarded as the Poincar\'e duals to the cohomology duals to the 2nd integral homology classes in a suitable basis of a free subgroup of rank $l$ of $H_2(M;\mathbb{Z})$ respecting the basis. As the basis, we can take the set of all values of the isomorphism ${{q_f}_{\ast}}^{-1}:H_2(W_f;\mathbb{Z}) \rightarrow H_2(M;\mathbb{Z})$ at all elements of the basis of the subgroup of $H_2(W_f;\mathbb{Z})$ before. As done in the proof of our Theorem \ref{thm:5} and that of Main Theorem 1 of \cite{kitazawa3}, we investigate (generic) intersections of two submanifolds in $\{{q_f}^{-1}(Y_j)\}$ or (generic) self-intersections of single submanifolds there. The dimensions of these resulting submanifolds are $(m-2)+(m-2)-m=m-4$. 
	
	We discuss the integral homology class represented by some connected component of such a submanifold. Showing that the zero element $0 \in H_{m-4}(M;\mathbb{Z})$ is represented by such a connected submanifold completes our proof by virtue of the facts on the rank of $H_{2}(M;\mathbb{Z})$ and Poincar\'e duality. Most of these arguments are based on the proof of Main Theorem 1 of \cite{kitazawa3}.
	
	This ($m-4$)-dimensional submanifold is the disjoint union of finitely many copies of the product $S^1 \times S^{m-5}$ of an even number and ($m-4$)-dimensional homotopy spheres of an even number. The element $0$ is represented by each of the former manifolds diffeomorphic to $S^1 \times S^{m-1}$ due to the fact that $W_f$ is simply-connected. These manifolds are the preimages of circles smoothly embedded in ${\rm Int}\ Y$ and ${\rm Int}\ W_f$.
	
	Each of the ($m-4$)-dimensional spheres here
	is the preimage of a $1$-dimensional compact and connected manifold diffeomorphic to a closed interval and embedded smoothly and properly in $W_f$. In other words, the interior is embedded in ${\rm Int}\ W_f$, the boundary is in $\partial W_f$ and the preimage is regarded as the homotopy sphere of the domain of a special generic map into $\mathbb{R}$ in Theorem \ref{thm:1} (\ref{thm:1.0}). 
	
	Here we review a main ingredient of the proof of Main Theorem \ref{mthm:1} of \cite{kitazawa3} where we use a bit different notation from the original one. $W_f$ is simply-connected and $\partial W_f$ is connected. This yields a smooth homotopy $H_f:{S^{m-4}}^{\prime} \times [0,1] \rightarrow W_f$ from the composition of the original embedding of each ($m-4$)-dimensional homotopy sphere ${S^{m-4}}^{\prime}$ here into $M$ with $q_f$ to a constant map to a point in $\partial W_f$.
	
	For $0 \leq t \leq 1$, define $S_{f,H}(t)$ as the set of all points in ${S^{m-4}}^{\prime}$ such that the pairs of the points and $t$ are mapped to $\partial W_f$ by the homotopy.
	We can take the homotopy $H_f$ satisfying $S_{f,H}(t_1) \subset S_{f,H}(t_2)$ for $0<t_1<t_2<1$. We can also see that the original embedding of the ${S^{m-4}}^{\prime}$ into $M$ is null-homotopic.

	This shows that the cup product of two 2nd integral cohomology classes is always the zero element $0 \in H^4(M;\mathbb{Z})$ for any pair of 2nd integral cohomology classes of $H^2(M;\mathbb{Z})$. 
	
	This completes the proof.
	
\end{proof}
\begin{Thm}
	\label{thm:7}
	For a closed and simply-connected manifold $M$ of dimension $m \geq 4$ admitting a special generic map $f:M \rightarrow {\mathbb{R}}^{m-1}$ whose singular set $S(f)$ consists of exactly $l>0$ connected components, we have the following two.
	\begin{enumerate}
		\item \label{thm:7.1}
		 The rank of $H_2(M;\mathbb{Z})$ is greater than or equal to $l-1$. The sum of the rank of $H_2(W_f;\mathbb{Z})$ and $l-1$ and the rank of $H_2(M;\mathbb{Z})$ agree.
		\item \label{thm:7.2}
		In {\rm (}\ref{thm:7.1}{\rm )}, assume also that $H_2(W_f;\mathbb{Z})$ is of rank $0$. Then the cup product $c_1 \cup c_2$ is the zero element for any pair $c_1,c_2 \in H^2(M;\mathbb{Z})$ of 2nd integral cohomology classes. 
	\end{enumerate}

\end{Thm}
\begin{proof}
	We apply a method as in the proof of Theorem \ref{thm:6} where we abuse the notation such as $W$, $W_f$ and $q_f:M \rightarrow W_f$. 
	
   We prove (\ref{thm:7.1}).
	
	We consider the homology exact sequence for
		$(W,M)$ again
	$$\rightarrow H_3(W,M;\mathbb{Z}) \rightarrow H_2(M;\mathbb{Z}) \rightarrow H_2(W;\mathbb{Z}) \rightarrow H_2(W,M;\mathbb{Z}) \rightarrow$$
	and $H_3(W,M;\mathbb{Z})$ is isomorphic to $H^{m-2}(W;\mathbb{Z})$, $H^{m-2}(W_f;\mathbb{Z})$ and $H_1(W_f,\partial W_f;\mathbb{Z})$ by Poincar\'e duality theorem for $W$ and $W_f$ and Proposition \ref{prop:3}, saying that $W$ collapses to $W_f$. $H_1(W_f;\mathbb{Z})$ is the trivial group by Proposition \ref{prop:3} and the proof of Theorem \ref{thm:2}, $M$ and $W_f$ are simply-connected. This is also in our proof of Theorem \ref{thm:5} for example. The homology exact sequence for $(W_f,\partial W_f)$ implies that $H_1(W_f,\partial W_f;\mathbb{Z})$ is free and of rank $l-1$. $H_3(W,M;\mathbb{Z})$, $H^{m-2}(W;\mathbb{Z})$, $H^{m-2}(W_f;\mathbb{Z})$ and $H_1(W_f,\partial W_f;\mathbb{Z})$ are shown to be isomorphic to this free group by the fact shown before. $H^{m-1}(W;\mathbb{Z})$ and $H_2(W,\partial W;\mathbb{Z})$ are isomorphic by Poincar\'e duality theorem and trivial groups since $W$ is a compact ($m-1$)-dimensional manifold smoothly immersed into ${\mathbb{R}}^{m-1}$ and collapses to $W_f$, having the homotopy type of a polyhedron whose dimension is smaller than $m-1$, by Proposition \ref{prop:3}. 
	
	We can take $l-1$ $1$-dimensional compact and connected manifolds diffeomorphic to a closed interval which are smoothly and disjointly embedded in $W_f$ and enjoy the following properties.
	\begin{itemize}
		\item There exists a basis of $H_1(W_f,\partial W_f;\mathbb{Z})$ and each of the elements of the basis is represented by one of the $l-1$ $1$-dimensional manifolds.
		\item For each of the $1$-dimensional manifolds, the interior is in the interior ${\rm Int}\ W_f$ and the boundary, consisting of exactly two points, is in the boundary $\partial W_f$. In other words, they are properly embedded in $W_f$.
		\item There exists exactly one connected component $C_0$ of the boundary $\partial W_f$ and this contains exactly one point of the boundary of each of these $l-1$ $1$-dimensional manifolds.
		\item For arbitrary distinct two $1$-dimensional manifolds of these $l-1$ manifolds, choose one suitable point in the boundary of each of the two, they are in distinct connected components of $\partial W_f$ which are not the connected component $C_0$ just before.  
		\item The preimage of each of these $l-1$ $1$-dimensional manifolds is regarded as the $2$-dimensional sphere $S^{2,j}$ of the domain of a special generic map into $\mathbb{R}$.
	\end{itemize}
	$H^1(W_f,\partial W_f;\mathbb{Z})$ is isomorphic to $H_1(W_f,\partial W_f;\mathbb{Z})$ and $H_{m-2}(W_f;\mathbb{Z})$ is isomorphic to these groups by Poincar\'e duality theorem. 
	We can also see the following three by the observation of 
	the intersections and Poincar\'e duality.
	\begin{itemize}
		\item All integral homology classes represented by the $l-1$ $2$-dimensional spheres in $\{S^{2,j}\}$ form a basis of a subgroup of $H_2(M;Z)$.
		\item Consider the integral homology class represented by the preimage ${q_f}^{-1}(C)$ of a connected component $C$ of the $l-1$ connected components of $\partial W_f$ different from the connected component $C_0$. The set of all $l-1$ homology classes here form a basis of a subgroup of $H_{m-2}(M;Z)$.
		\item The subgroup generated by the former basis is mapped to the trivial subgroup of $H_2(W;\mathbb{Z})$ by the homomorphism ${i_M}_{\ast}$ induced by the canonical inclusion $i_M:M \rightarrow W$: note that this homomorphism is also in the homology exact sequence for $(W,M)$.
		\end{itemize}
	More precisely, for each $S^{2,j}$ in $\{S^{2,j}\}$, exactly one connected component $C$ of the $l-1$ connected components of $\partial W_f$ different from the connected component $C_0$ and $S^{2,j}$ is not disjoint and they have exactly one common point. This yields the two bases in the first two facts here. Furthermore, the subgroup generated by the first basis is mapped to the trivial subgroup of $H_2(W;\mathbb{Z})$ as in the third fact here since $W$ collapses to $W_f$ via a PL map $r_f$ satisfying $q_f=r_f \circ i_M$ in Proposition \ref{prop:3}.
	
  Remember that $H_2(W,M;\mathbb{Z})$ has been shown to be the trivial group before.
  $H_3(W,M;\mathbb{Z})$ has been shown to be isomorphic to $H^{m-2}(W;\mathbb{Z})$ and $H^{m-2}(W_f;\mathbb{Z})$, free and of rank $l-1$.
   Thus the sum of the rank of $H_3(W,M;\mathbb{Z})$ and that of $H_2(W;\mathbb{Z})$ is same as that of $H_2(M;\mathbb{Z})$. $W$ collapses to $W_f$ and $H_2(W;\mathbb{Z})$ is replaced by $H_2(W_f;\mathbb{Z})$. The rank of $H_3(W,M;\mathbb{Z})$ is $l-1$.

This completes the proof of (\ref{thm:7.1}).	
  
  We show (\ref{thm:7.2}). 

$H_2(M;\mathbb{Z})$ is of rank $l-1$ in the case where the rank of $H_2(W_f;\mathbb{Z})$ is $0$. $H_2(M;\mathbb{Z})$ is represented as the internal direct sum of the free commutative subgroup having the basis consisting of exactly $l-1$ elements represented by the $l-1$ spheres in $\{S^{2,j}\}$ before and a finite commutative group. We can consider the Poincar\'e duals to the cohomology duals to the elements in the basis respecting it and regard that they are represented by the $l-1$ connected components of the singular set of $f$ different from $C_0$ as before. Note that the cohomology duals to the elements in the basis of the free commutative subgroup of $H_2(M;\mathbb{Z})$ respecting it form a basis of $H^2(M;\mathbb{Z})$. 
  
 We can move each connected component of the $l-1$ connected components of the singular set of $f$ different from $C_0$ by a smooth isotopy so that the resulting image and the original image or the connected component of the singular set are disjoint in $M$. In other words, we can have a {\it self-intersection} which is empty. We can explain about this fact as in the end of the proof of Theorem \ref{thm:5}. $H_2(W_f;\mathbb{Z})$ is a finite group and $H^2(W_f;\mathbb{Z})$ is the trivial group. 

We can argue similarly to complete the proof of (\ref{thm:7.2}).

This completes the proof.
\end{proof}
\begin{proof}[A proof of Theorem \ref{thm:4} via Theorems \ref{thm:6} and \ref{thm:7}]
	Suppose that ${\mathbb{C}P}^3$ admits a special generic map $f$ into ${\mathbb{R}}^5$. Then, the singular set is not connected and consists of exactly two connected components by Theorem \ref{thm:6} and Theorem \ref{thm:7} (\ref{thm:7.1}). Furthermore, the rank of $H_2({\mathbb{C}P}^3;\mathbb{Z})$ is $1$ and the rank of $H_2(W_f;\mathbb{Z})$ is $0$. This contradicts Theorem \ref{thm:7} (\ref{thm:7.2}). 
\end{proof}
\subsection{The third method.}
We respect philosophy of so-called Rokhlin's theorem on restrctions on the integral cohomology rings and the {\it signatures} of $4$-dimensional closed and oriented smooth manifolds. The {\it signature} of a compact and oriented manifold is defined uniquely from the cohomology ring and the orientation. The following theorem is our main theorem of this subsection. 

\begin{Thm}
\label{thm:8}
If a $6$-dimensional closed and simply-connected spin manifold $M$ has an element $u \in H^2(M;\mathbb{Z})$ whose square $u \cup u \in H^4(M;\mathbb{Z})$ is not divisible by $2$, then it admits no special generic maps into ${\mathbb{R}}^5$.
\end{Thm}

\begin{Prop}
	\label{prop:4}
For a closed and simply-connected manifold $M$ of dimension $m \geq 4$ admitting a special generic map $f:M \rightarrow {\mathbb{R}}^{m-1}$ whose singular set $S(f)$ consists of exactly $l>0$ connected components we have the following two where we abuse the notation such as $W_f$ similarly.
\begin{itemize}
	\item The rank of $H_2(M;\mathbb{Z}/2\mathbb{Z})$ is greater than or equal to $l-1$. 
	\item The sum of the rank of $H_2(W_f;\mathbb{Z}/2\mathbb{Z})$ and $l-1$ and the rank of $H_2(M;\mathbb{Z}/2\mathbb{Z})$ agree.
\end{itemize}
\end{Prop}

We also have the following proposition.

\begin{Prop}
	\label{prop:5}
Consider the situation of Proposition \ref{prop:4}.

We can decompose $H_2(W_f;\mathbb{Z})$ into a suitable internal direct sum $G_{W_f,{\rm Free}} \oplus G_{W_f,{\rm Finite}}$ where the former summand $G_{W_f,{\rm Free}}$ and the latter summand $G_{W_f,{\rm Finite}}$ are a free commutative subgroup and a finite commutative subgroup respectively.
Let $H_{2,\mathbb{Z}}(W_f;\mathbb{Z}/2\mathbb{Z}) \subset H_2(W_f;\mathbb{Z}/2\mathbb{Z})$ be the image of the homomorphism from $G_{W_f,{\rm Free}}$ into
$H_2(W_f;\mathbb{Z}/2\mathbb{Z})$ obtained canonically by considering the canonical homomorphism or the quotient map from $\mathbb{Z}$ onto $\mathbb{Z}/2\mathbb{Z}$. 

We can decompose $H_2(M;\mathbb{Z})$ into a suitable internal direct sum $G_{M,{\rm Free}} \oplus G_{M,{\rm Finite}}$ where the former summand $G_{M,{\rm Free}}$ and the latter summand $G_{M,{\rm Finite}}$ are a free commutative subgroup and a finite commutative subgroup respectively.

Let $H_{2,\mathbb{Z}}(M;\mathbb{Z}/2\mathbb{Z}) \subset H_2(M;\mathbb{Z}/2\mathbb{Z})$ be the image of the homomorphism from $G_{M,{\rm Free}}$ into $H_2(M;\mathbb{Z}/2\mathbb{Z})$ obtained similarly.	

Then we have the following properties.

\begin{itemize}
	\item The rank of $H_{2,\mathbb{Z}}(M;\mathbb{Z}/2\mathbb{Z})$ is greater than or equal to $l-1$. 
	\item The sum of the rank of $H_{2,\mathbb{Z}}(W_f;\mathbb{Z}/2\mathbb{Z})$ and $l-1$ and the rank of $H_{2,\mathbb{Z}}(M;\mathbb{Z}/2\mathbb{Z})$ agree.
\end{itemize}
\end{Prop}
We can prove these two propositions as Theorem \ref{thm:7} (\ref{thm:7.1}). We omit their proofs. For example, as specific cases, in the case where the integral homology groups have no torsions, we have these facts by changing the coefficient rings.
\begin{proof}[A proof of Theorem \ref{thm:8}]

In Theorem \ref{thm:8}, Propositions \ref{prop:4} and \ref{prop:5} are key ingredients. We abuse the notation there and we also abuse the notation in Proposition \ref{prop:3} for example. \\
\ \\
STEP 1 Construction of a basis of $H_{2,\mathbb{Z}}(M;\mathbb{Z}/2\mathbb{Z})$. \\
\ \\
First we remember a main argument in the proof of Theorem \ref{thm:5} again.  $H_{2,\mathbb{Z}}(W_f;\mathbb{Z}/2\mathbb{Z})$ has a basis and each element $e_{2,j}$ of the basis is represented by a smooth embedding $e_{2,j}:S^2 \rightarrow {\rm Int}\ W_f \subset W_f$ where we abuse the notation $e_{2,j}$ for the embedding. We can also have the following objects and argue in the following way.
\begin{itemize}
	\item There exists a connected component of $C_0$ of the boundary $\partial W_f$ where we abuse the notation in the proof of Theorem \ref{thm:7}.
	\item There exists a smooth embedding $s_{2,j}:S^2 \rightarrow M$ enjoying the following properties as in the proof of Theorem \ref{thm:5}.
	\begin{itemize}
		\item The composition of this with $q_f:M \rightarrow W_f$ is smoothly homotopic to the original embedding $e_{2,j}$ of $S^2$ into ${\rm Int}\ W_f \subset W_f$. Let ${e_{2,j}}^{\prime}:S^2 \rightarrow W_f$ denote this.
		\item $s_{2,j}$ is a kind of variants of sections of the original bundle over $e_{2,j}(S^2) \subset {\rm Int}\ W_f$, whose projection is given by $q_f$. $W_f$ is a $5$-dimensional manifold smoothly immersed into ${\mathbb{R}}^5$ and spin as in the proof of Theorem \ref{thm:5}. $S^2$ is spin (by Theorem \ref{thm:3} for example). A normal bundle of (the embedding of) the $2$-dimensional sphere in ${\rm Int}\ W_f$ is shown to be trivial by a fundamental argument on linear bundles over $S^2$.  
		Thanks to the assumption that $M$ is spin together with a related fundamental argument, we can regard that the intersection of ${e_{2,j}}^{\prime}(S^2)$ and the boundary $\partial W_f$ is a finite set consisting of points of an even number and a subset of $C_0$ for a suitably chosen embedding $s_{2,j}$. 
	\end{itemize}
\item We have the cohomology dual ${e_{2,j}}^{\ast} \in H^2(W_f;\mathbb{Z}/2\mathbb{Z})$ to each element $e_{2,j} \in H_{2,\mathbb{Z}}(W_f;\mathbb{Z}/2\mathbb{Z}) \subset H_2(W_f;\mathbb{Z}/2\mathbb{Z})$ of the basis $\{e_{2,j}\}$ of $H_{2,\mathbb{Z}}(W_f;\mathbb{Z}/2\mathbb{Z})$ respecting it. The Poincar\'e dual ${\rm PD}({e_{2,j}}^{\ast}) \in H_3(W_f,\partial W_f;\mathbb{Z}/2\mathbb{Z})$ to ${e_{2,j}}^{\ast}$ is represented by a $3$-dimensional, compact, connected and orientable manifold $Y_{2,j}$ smoothly embedded in $W_f$. We can also assume the following properties as the proof of Theorem \ref{thm:5} presents.
\begin{itemize}
	\item The boundary $\partial Y_{2,j}$ is embedded in the boundary $\partial W_f$ and the interior ${\rm Int}\ Y_{2,j}$ is embedded in the interior ${\rm Int}\ W_{2,j}$. In other words, $Y_{2,j}$ is embedded properly in $W_f$. 
	\item ${q_f}^{-1}(Y_{2,j})$ can be regarded as a $4$-dimensional closed, connected and orientable manifold and regarded as the manifold of the domain of a special generic map into a $3$-dimensional, non-compact, connected and orientable manifold with no boundary.
\end{itemize}
 
\end{itemize}
As in the proof of Theorem \ref{thm:7}, we can take $l-1$ $1$-dimensional manifolds diffeomorphic to a closed interval which are smoothly and disjointly embedded in $W_f$ and we can argue in the following way.
\begin{itemize}
	\item There exists a basis of $H_1(W_f,\partial W_f;\mathbb{Z})$ and each of the elements of the basis is represented by one of the $l-1$ $1$-dimensional manifolds.
	\item For each of the $1$-dimensional manifolds, the interior is in the interior ${\rm Int}\ W_f$ and the boundary, consisting of exactly two points, is in the boundary $\partial W_f$. In other words, it is embedded properly in $W_f$.
	\item The connected component $C_0$ of the boundary $\partial W_f$, defined before, contains exactly one point of the boundary of each of these $l-1$ $1$-dimensional manifolds.
	\item For arbitrary distinct two $1$-dimensional manifolds of these $l-1$ manifolds, choose one suitable point in the boundary of each of the two, they are in distinct connected components of $\partial W_f$ which are not the connected component $C_0$.  
	\item The preimage of each of these $l-1$ $1$-dimensional manifolds is regarded as the $2$-dimensional sphere $S^{2,j}$ of the domain of a special generic map into $\mathbb{R}$ where we abuse the notation $S^{2,j}$ in the proof of Theorem \ref{thm:7}. Furthermore, all classes represented by the $l-1$ spheres in $\{S^{2,j}\}$ form a basis of a subgroup of $H_{2,\mathbb{Z}}(M;\mathbb{Z}/2\mathbb{Z}) \subset H_{2}(M;\mathbb{Z}/2\mathbb{Z})$ where the coefficient ring is $\mathbb{Z}/2\mathbb{Z}$. Let $e_{S^{2,j}} \in H_{2,\mathbb{Z}}(M;\mathbb{Z}/2\mathbb{Z}) \subset H_{2}(M;\mathbb{Z}/2\mathbb{Z})$ denote the element represented by $S^{2,j}$, an embedding of $S^2$.
\end{itemize}
    For the basis $\{e_{2,j}\}$ of $H_{2,\mathbb{Z}}(W_f;\mathbb{Z}/2\mathbb{Z})$ and each element, we have a smooth embedding $s_{2,j}:S^2 \rightarrow M$ making $q_f \circ s_{2,j}:M \rightarrow W_f$ and the given embedding $e_{2,j}:S^2 \rightarrow {\rm Int}\ W_f \subset W_f$ smoothly homotopic. Let $e_{s,2,j}$ denote the element of $H_{2,\mathbb{Z}}(M;\mathbb{Z}/2\mathbb{Z})$ represented by the embedding $s_{2,j}$. The disjoint union of the set $\{e_{S^{2,j_1}}\}$ and $\{e_{s,2,j_2}\}$ form a basis of $H_{2,\mathbb{Z}}(M;\mathbb{Z}/2\mathbb{Z})$ by virtue of Proposition \ref{prop:5}. Note that the subgroup generated by the basis $\{e_{S^{2,j_1}}\}$ is mapped to the trivial group by ${i_M}_{\ast}$ and ${q_f}_{\ast}={r_f}_{\ast} \circ {i_M}_{\ast}$ and that the subgroup generated by the basis $\{e_{s,2,j_2}\}$ is mapped to $H_{2,\mathbb{Z}}(W_f;\mathbb{Z}/2\mathbb{Z})$ by ${q_f}_{\ast}={r_f}_{\ast} \circ {i_M}_{\ast}$, which is regarded as an isomorphism from this subgroup to $H_{2,\mathbb{Z}}(W_f;\mathbb{Z}/2\mathbb{Z})$. We can define the cohomology dual ${e_{S^{2,j_1}}}^{\ast} \in H^2(M;\mathbb{Z}/2\mathbb{Z})$ to $e_{S^{2,j_1}}$ and the cohomology dual ${e_{s,2,j_2}}^{\ast} \in H^2(M;\mathbb{Z}/2\mathbb{Z})$ to $e_{s,2,j_2}$ respecting the bases. \\
    \ \\
STEP 2  The Poincar\'e duals to elements of the basis $\{e_{S^{2,j_1}}\} \sqcup \{e_{s,2,j_2}\}$ of $H_{2,\mathbb{Z}}(M;\mathbb{Z}/2\mathbb{Z})$. \\
 \ \\     
    Let 
    $e_{4,2,C} \in H_4(M;\mathbb{Z}/2\mathbb{Z})$ denote the homology class represented by a connected component $C$ of the singular set of $f$. Let $e_{4,2,Y_j} \in H_4(M;\mathbb{Z}/2\mathbb{Z})$ denote the homology class represented by the $4$-dimensional closed, connected and orientable manifold ${q_f}^{-1}(Y_{2,j})$.

    We can do so that ${q_f}^{-1}(Y_{2,j_1}) \bigcap S^{2,j_2}$ is empty for any $j_1$ and $j_2$. $Y_{2,j_1}$ is $3$-dimensional and embedded smoothly and properly in $W_f$.
     $S^{2,j_2}$ is regarded as the preimage of a $1$-dimensional manifold diffeomorphic to a closed interval and embedded smoothly and properly in $W_f$. $W_f$ is $5$-dimensional. The relation $3+1=4<5$ implies this. 
     
     $S^{2,j} \bigcap C$ is empty for every connected component $C$ of the singular set except the connected component $C_0$ and one another connected component $C_{j^{\prime}}$. For each connected component $C$ of these two, this set is a one-point set. For distinct spheres $S^{2,j_1}$ and $S^{2,j_2}$, the connected components $C_{{j_1}^{\prime}}$ and $C_{{j_2}^{\prime}}$ which are not $C_0$ are distinct. This reviews the definitions.
     
      Thus $e_{4,2,C_{j^{\prime}}}$ is the Poincar\'e dual to ${e_{S^{2,j}}}^{\ast}$.
      
      We remember some properties of the embedding 
      $s_{2,j}:S^2 \rightarrow M$. 
      $s_{2,j}(S^2) \bigcap C$ is empty for every connected component $C$ of the singular set except the connected component $C_0$. $s_{2,j}(S^2) \bigcap C_0$ is a finite set whose size is an even number. $s_{2,j_1}(S^2) \bigcap {q_f}^{-1}(Y_{2,j_2})$ can be regarded as a finite subset of $M$. The size of this finite set is odd if and only if $j_1=j_2$. Here, we consider Poincar\'e duality for $W_f$ and based on this, we consider Poincar\'e duality for the manifold $M$.
      
      This implies that $e_{4,2,Y_j}$ is the Poincar\'e dual to ${e_{s,2,j}}^{\ast}$.
            \\
      \ \\
STEP 3  The square $u^{\ast} \cup u^{\ast}$ of the cohomology dual $u^{\ast} \in H^2(M;\mathbb{Z}/2\mathbb{Z})$ to each element $u$ in the basis $\{e_{S^{2,j_1}}\} \sqcup \{e_{s,2,j_2}\}$ of $H_{2,\mathbb{Z}}(M;\mathbb{Z}/2\mathbb{Z})$ respecting it. \\    
\ \\
As we have done in several scenes in the present paper, we consider the Poincar\'e duals and investigate self-intersections.

Remember arguments on the square ${e_s}^{\ast} \cup {e_s}^{\ast}$ in the proof of Theorem \ref{thm:5}. We can argue similarly to know that the square ${e_{s,2,j}}^{\ast} \cup {e_{s,2,j}}^{\ast}$ is the zero element of $H^4(M;\mathbb{Z}/2\mathbb{Z})$. More precisely, if we consider homology classes where the coefficient ring is $\mathbb{Z}$, then we have a $4$th integral cohomology class or element of $H^4(M;\mathbb{Z})$ which is divisible by $2$. 

According to fundamental propositions in section 3 such as Lemma 3.7 in \cite{saeki}, the tangent bundle of $C$ and the trivial linear bundle whose fiber is diffeomorphic to $\mathbb{R}$ is trivial for any connected component $C$ of the singular set $S(f)$. The tangent bundle of $C$ is spin by a fundamental argument. $M$ is assumed to be spin. We can see that the linear bundle whose fiber is the $2$-dimensional unit disk $D^2$ in Proposition \ref{prop:2} (\ref{prop:2.4}) is also spin by a fundamental argument. About self-intersections here, we can argue as in the previous case and the square
${e_{S^{2,j}}}^{\ast} \cup {e_{S^{2,j}}}^{\ast}$ is also the zero element. If we consider the case where the coefficient ring
is $\mathbb{Z}$, we have a similar result.

This completes the proof.    

\end{proof}
We easily have Theorem \ref{thm:4} as a corollary to Theorem \ref{thm:8}.

We present meaningful related examples.

\begin{Ex}
	\label{ex:2}
	\begin{enumerate}
		\item \label{ex:2.1}
		We generalize the construction of a special generic map, just after Proposition \ref{prop:2} and just before Example \ref{ex:1}, explicitly. For more precise expositions, see the original paper \cite{saeki} and \cite{kitazawa,kitazawa2,kitazawa3,kitazawa4} by the author for example.
		
		Let $(m,n)=(6,5)$ and we can take $W_f:=S^2 \times S^2 \times D^1$ smoothly embedded in ${\mathbb{R}}^5$ by $\bar{f_0}$.
		
		By considering two trivial bundles in the fifth property, we have a special generic map $f_0$ from $M_0:=S^2 \times S^2 \times S^2$ into ${\mathbb{R}}^5$. 
		Linear bundles whose fibers are circles and whose structure groups consist of linear transformations preserving the orientations of the fibers have been discussed shortly and explicitly in our paper. We apply fundamental arguments on such linear bundles and we can replace the two trivial bundles by suitable linear bundles to obtain a special generic map $f_0$ from $M_0:=(S^2 \tilde{\times} S^2) \times S^2$ into ${\mathbb{R}}^5$ where $S^2 \tilde{\times} S^2$ denotes the total space of a linear bundle whose fiber is diffeomorphic to $S^2$ and which is not trivial.  
		The total space is known to be not spin and diffeomorphic to a manifold represented as a connected sum of two copies of ${{\mathbb{C}}P}^2$ where we consider ${\mathbb{C}P}^2$ as a canonically oriented manifold, reverse the orientation for exactly one of the two copies and consider the connected sum in the smooth category. This new manifold $M_0$ is not spin. This does not satisfy the assumption of Theorem \ref{thm:5} or Theorem \ref{thm:6}.
		
				\item \label{ex:2.2}\cite{choimasudasuh} is on classifications of so-called ({\it generalized}) {\it Bott manifolds}. Complex projective spaces are simplest generalized Bott manifolds. Such manifolds are also so-called {\it toric} manifolds.
				A {\it Bott manifold} is, in short, a closed and simply-connected smooth manifold obtained by a finite iteration of taking linear bundles whose fibers are the $2$-dimensional unit sphere of a certain class which respects the complex structures of the manifolds. The product of finitely many copies of $S^2$ is always a Bott manifold. $S^2 \tilde{\times} S^2$ and the $6$-dimensional manifold ${M_0}^{\prime}:=(S^2 \tilde{\times} S^2) \times S^2$, discussed in {\rm (}\ref{ex:2.1}{\rm )}, are also Bott manifolds.
				
				In section 7 of \cite{choimasudasuh}, from three integers "$a$","$b$","$c$", a $6$-dimensional Bott manifold is uniquely obtained. By taking the first two integers as arbitrary odd numbers and the last integer an arbitrary even number, we have a spin manifold satisfying the assumption of Theorem \ref{thm:8}. From Lemma 7.3 there, we can see that we have a family of countably many such $6$-dimensional Bott manifolds such that distinct manifolds are not homeomorphic. They are distinguished by the 1st Pontrjagin classes. The 1st Pontrjagin class of the $6$-dimensional (oriented) Bott manifold is represented by $4k$ times an element of some $4$th integral cohomology class which is not divisible by any integer greater than $1$. Furthermore, we can realize every integer $k$ by such $6$-dimensional spin Bott manifolds satisfying the assumption of Theorem \ref{thm:8}.
	\end{enumerate}
	
\end{Ex}
Note that $6$-dimensional manifolds in Example \ref{ex:2} do not have special generic maps into ${\mathbb{R}}^n$ for $n=1,2,3,4$ by Theorem \ref{thm:2}. For example, the new manifold $M_0$ in Example \ref{ex:2} (\ref{ex:2.1}) does not admit special generic maps into ${\mathbb{R}}^6$ by the theory \cite{eliashberg}. If the 1st Pontrjagin class of a (suitably oriented) manifold there is not the zero element, then it does not admit special generic maps into ${\mathbb{R}}^6$ by the same reason. 
Example \ref{ex:2} (\ref{ex:2.2}) completes the proof of Main Theorem \ref{mthm:2}.
 \section{Remarks.}
We close our paper by several remarks.

\begin{Rem}
	\label{rem:2}
	The $2$-dimensional complex projective space ${\mathbb{C}P}^2$ does not admit special generic maps into ${\mathbb{R}}^n$ for $n=1,2,3$ by Theorem \ref{thm:1} (\ref{thm:1.1}) and (\ref{thm:1.2}), Theorem \ref{thm:2}, Theorem \ref{thm:3} and Corollary \ref{cor:1} (\ref{cor:1.2}). It does not admit ones into ${\mathbb{R}}^4$, which follows from the theory \cite{eliashberg} or \cite{kikuchisaeki}.

	According to \cite{kikuchisaeki}, if a closed manifold whose Euler number is odd admits a special generic map into an Euclidean space, then the dimension of the Euclidean space must be $1$, $3$ or $7$. This theorem is, more generally, one for {\it fold} maps: the class of {\it fold} maps is a natural extension of the class of Morse functions and that of special generic maps. 
	
	A {\it fold} map is, in short, defined as a smooth map locally represented as the product maps of the Morse functions and the identity maps on open disks. \cite{golubitskyguillemin} presents some singularity theory on singular points of fold maps and more general generic smooth maps. \cite{thom2, whitney} are pioneering studies on fold maps or more general generic smooth maps into the plane on closed manifolds whose dimensions are greater than or equal to $2$. \cite{kitazawa0.1, kitazawa0.2, kitazawa0.3, kitazawa} are on {\it round} fold maps, defined as fold maps whose singular value sets are embedded concentric spheres and defined first by the author in \cite{kitazawa0.1, kitazawa0.2, kitazawa0.3}.  
\end{Rem}	

\begin{Rem}
	\label{rem:3}
	The existence and non-existence of fold maps on closed manifolds have been natural, important, difficult and attractive problems. Morse functions always exist. Fold maps into ${\mathbb{R}}^2$ exist on a closed manifold whose dimension is greater than or equal to $2$ if the Euler number is even.
	Eliashberg's theory \cite{eliashberg,eliashberg2} has solved problems in considerable cases. The triviality of the Whitney sum of the tangent bundle and a trivial linear bundle whose fiber is $\mathbb{R}$ is a strong sufficient condition for the existence. Later \cite{ando} has given generalized answers. 
	
	For projective spaces, such problems are still difficult to solve in considerable cases. For example, we do not know the existence or the non-existence of fold maps into ${\mathbb{R}}^5$ on ${\mathbb{C}P}^3$. We know the existence for $n=1,2,3,4$. 
	For this, see also construction of explicit fold maps on the total spaces of smooth bundles over standard spheres in \cite{kitazawa0.1,kitazawa0.2,kitazawa} for example.
	For ${\mathbb{C}P}^2$, we know the non-existence of fold maps into ${\mathbb{R}}^n$ for $n=2,3,4$. For the case $n=3$, see also \cite{saeki3}.
	
	For related studies, see \cite{ohmotosaekisakuma,sadykovsaekisakuma}, for example.

\end{Rem}
\begin{Rem}
\label{rem:4}
We can know the non-existence of special generic maps on the $k$-dimensional complex projective space ${\mathbb{C}P}^k$ into ${\mathbb{R}}^n$ for $1 \leq n \leq 2k-2$ and $n=2k$ where $k \geq 2$ by Theorems \ref{thm:2} and \ref{thm:3}. We can know the non-existence of special generic maps on ${\mathbb{C}P}^k$ into ${\mathbb{R}}^n$ for $1 \leq n \leq 2k$ where $k \geq 2$ is an even integer. This is due to Proposition \ref{prop:3}, implying that closed manifolds admitting special generic maps into lower dimensional spaces must bound compact manifolds, whereas ${{\mathbb{C}}P}^k$ does not enjoy this property for any even integer $k \geq 2$. We do not know the (non-)existence of special generic maps on ${\mathbb{C}P}^k$ into ${\mathbb{R}}^{2k-1}$ where $k \geq 5$ is an odd integer.
\end{Rem}

	Related to our remarks here, we present our additional main result.
	\begin{MainThm}
		\label{mthm:3}
For any $4$-dimensional smooth closed and simply-connected $M^{\prime}$ having an element $u_0 \in H^2(M^{\prime};\mathbb{Z})$ whose square $u_0 \cup u_0 \in H^4(M^{\prime};\mathbb{Z})$ is not divisible by $2$, we have a $6$-dimensional closed and simply-connected manifold $M$ for Theorem \ref{thm:8} enjoying the following properties. 
		\begin{enumerate}
\item $M$ is the total space of a linear bundle over $M^{\prime}$ whose fiber is the $2$-dimensional unit sphere. $H_j(M;\mathbb{Z})$ and $H_j(M^{\prime};\mathbb{Z})$ are free for any integer $j$. 
$H_j(M;\mathbb{Z})$ and $H^j(M;\mathbb{Z})$ are isomorphic and $H_j(M^{\prime};\mathbb{Z})$ and $H^j(M^{\prime};\mathbb{Z})$ are isomorphic for any integer $j$. $H_j(M^{\prime};\mathbb{Z})$ is trivial for $j=1,3$. The rank of $H_j(M;\mathbb{Z})$ is as follows.
\begin{enumerate}
\item The rank is $0$ for $j \neq 0,2,4,6$.
\item The rank is $1$ for $j=0,6$.
\item The rank is the sum of the rank of $H_2(M^{\prime};\mathbb{Z})$ and $1$.
\end{enumerate}
\item There exists a homomorphism $\phi:H^{\ast}(M;\mathbb{Z}) \rightarrow H^{\ast}(M^{\prime};\mathbb{Z})$ simply seen as one between graded modules over $\mathbb{Z}$ such that for any ${u}^{\prime} \in H^{\ast}(M^{\prime};\mathbb{Z})$ we have a suitable element $u \in H^{\ast}(M;\mathbb{Z})$ satisfying $\phi(u)={u}^{\prime}$. 
			\item $M$ admits a fold map into ${\mathbb{R}}^n$ for $n=1,2,3,4,5,6$.
			\item $M$ admits a special generic map into ${\mathbb{R}}^n$ if and only if $n=6$.
		\end{enumerate}
	\end{MainThm}
\begin{proof}
Thanks to the classical theory of smooth immersions, we can smoothly immerse $M^{\prime}$ into ${\mathbb{R}}^7$. We have a so-called {\it normal bundle} as a linear bundle over $M^{\prime}$ whose fiber is the unit disk $D^3$.
We define $M$ as the boundary of the total space. The first property is on algebraic topological properties and follows from fundamental arguments on linear bundles and algebraic topology.
We discuss the second property. For  any ${u}^{\prime} \in H^{j}(M^{\prime};\mathbb{Z})$, we can have a unique element Poincar\'e dual to it (where the manifolds are suitably oriented). For $0 \leq j \leq 4$, it is an element of $H_{4-j}(M^{\prime};\mathbb{Z})$ and represented by a smooth compact and connected orientable submanifold $S_{u^{\prime}}$ with no boundary. We consider the restriction $S_u$ of the bundle $M$ over $M^{\prime}$ to the submanifold $S_{u^{\prime}}$. We can obtain a unique element Poincar\'e dual to the element represented by $S_u$. This is our desired $u \in H^{\ast}(M;\mathbb{Z})$ to have our desired homomorphism $\phi:H^{\ast}(M;\mathbb{Z}) \rightarrow H^{\ast}(M^{\prime};\mathbb{Z})$ enjoying the relation $\phi(u)={u}^{\prime}$.
The third property follows from the fact that the tangent bundle of $M$ is easily seen to be trivial by fundamental arguments on differential topology. \cite{eliashberg,eliashberg2} complete the proof of this property. Note that in the third property here, in the case $n=6$, a fold map is a special generic map. For the fourth property, in the case where the dimension $n$ of the Euclidean space of the target is $n=1,2,3,4$, we apply Theorems \ref{thm:1} and \ref{thm:2} to complete its proof. The second property, the existence of $u_0$ and the trivial tangent bundle show that $M$ is for Theorem \ref{thm:8}. Theorem \ref{thm:8} with the third property for the dimension $n=6$ of the Euclidean space of the target completes the proof of the fourth property here. This completes the proof.
\end{proof}
\begin{MainThm}
\label{mthm:4}
A $6$-dimensional generalized Bott manifold does not admit special generic maps into ${\mathbb{R}}^n$ for $n=1,2,3,4$.
A $6$-dimensional generalized Bott manifold which is also spin admits special generic maps into ${\mathbb{R}}^5$ if and only if it is represented as the total space of a linear bundle over $S^2 \times S^2$ whose fiber is diffeomorphic to $S^2$.
\end{MainThm}
\begin{proof}
We do not explain about generalized Bott manifolds rigorously. 
This is a manifold obtained by a finite iteration of taking smooth bundles whose fibers are complex projective spaces. More precisely, we need to consider specific bundles: we need to consider a Whitney sum of bundles whose fibers are the complex line $\mathbb{C}$ and whose structure groups consist of complex linear transformations and consider the projectivization. 
For related exposition, see \cite{choimasudasuh} for example.

The former statement follows from Theorem \ref{thm:1} or Theorem \ref{thm:2}.

We discuss the latter statement. Theorem \ref{thm:8} restricts the manifold to be the total space of linear bundle over $S^2 \times S^2$ whose fiber is diffeomorphic to $S^2$.
Conversely, for the total space of a linear bundle over $S^2 \times S^2$ whose fiber is diffeomorphic to $S^2$ which may not be spin or a generalized Bott manifold, we can reconstruct a special generic map as presented just after Proposition \ref{prop:2}.
More precisely, we prepare a smooth embedding $\bar{f}:S^2 \times S^2 \times [-1,1] \rightarrow {\mathbb{R}}^5$ and apply the presented method in a bit generalized manner, which is presented in the original article: we consider bundles which may not be trivial.
This completes the proof of the latter statement.

This completes the proof. 
\end{proof}
\section{Acknowledgement.}
The author would like to thank Osamu Saeki for an interesting informal seminar on \cite{kitazawa3} and the present paper, which applies some arguments in \cite{kitazawa3} with new ones.
The author was a member of the project supported by JSPS KAKENHI Grant Number JP17H06128 "Innovative research of geometric topology and singularities of differentiable mappings"
(Principal investigator: Osamu Saeki). The present study was supported by the project. The author was also a member of the project JSPS KAKENHI Grant Number JP22K18267 "Visualizing twists in data through monodromy" (Principal Investigator: Osamu Saeki). Our study thanks this project for supports. The present study is also supported by the project. At present, the author works at Institute of Mathematics for Industry (https://www.jgmi.kyushu-u.ac.jp/en/about/young-mentors/). This is closely related to the present study. The author is also a researcher at Osaka Central
Advanced Mathematical Institute (OCAMI researcher), supported by MEXT Promotion of Distinctive Joint Research Center Program JPMXP0723833165. He is not working there. However this helps the studies and our study also thanks this. \\

We declare that data essentially supporting our present study are all in the present paper.
\end{document}